\documentclass[12pt]{article}

\usepackage[
    top=1in,
    bottom=1.2in,
    left=1.35in,
    right=1.35in
]{geometry}
\usepackage{xcolor}
\usepackage{amsmath,amsthm,amssymb}
\usepackage{xurl}
\usepackage{float}
\usepackage{needspace}
\usepackage{enumitem}
\usepackage{tikz}
\usepackage[hidelinks]{hyperref}

\setlist[itemize]{itemsep=3pt,parsep=0pt,topsep=3pt,partopsep=0pt}
\setlist[enumerate]{itemsep=3pt,parsep=0pt,topsep=3pt,partopsep=0pt}

\let\oldthebibliography\thebibliography
\renewcommand{\thebibliography}[1]{%
  \oldthebibliography{#1}%
  \setlength{\itemsep}{4pt}%
  \setlength{\parskip}{0pt}%
  \setlength{\parsep}{0pt}%
}

\tikzset{
    vtx/.style={draw,circle,very thin,white,fill=black,inner sep=0pt,text width=8pt},
    edg/.style={ultra thick},
    ellip/.style={very thin,black,fill=white}
}
\pgfdeclarelayer{fore}
\pgfdeclarelayer{forefore}
\pgfsetlayers{main,fore,forefore}

\newtheorem{theorem}{Theorem}[section]
\newtheorem{lemma}[theorem]{Lemma}
\newtheorem{corollary}[theorem]{Corollary}
\newtheorem{conjecture}[theorem]{Conjecture}
\newtheorem{proposition}[theorem]{Proposition}
\newtheorem{observation}[theorem]{Observation}
\newtheorem{question}[theorem]{Question}
\newtheorem{remark}[theorem]{Remark}
\newtheorem*{claim}{Claim}
\theoremstyle{definition}
\newtheorem{example}[theorem]{Example}
\theoremstyle{plain}

\newenvironment{claimproof}
  {\begin{proof}}
  {\end{proof}}

\title{Bounded diameter covering of $2$-colored complete bipartite graphs}
\author{Louis DeBiasio\thanks{Department of Mathematics, Miami University, Oxford, OH. \texttt{debiasld@miamioh.edu}. Research supported in part by AMS-Simons Research Enhancement Grants for PUI Faculty GR001015.} \and Andr\'as Gy\'arf\'as\thanks{Alfr\'ed R\'enyi Institute of Mathematics, Budapest, P.O. Box 127, Budapest, Hungary, H-1364.
\texttt{gyarfas.andras@renyi.hu}, \texttt{sarkozy.gabor@renyi.hu}}
  \and G\'{a}bor N. S\'ark\"ozy\footnotemark[2]
\thanks{Computer Science Department, Worcester Polytechnic Institute, Worcester, MA (corresponding author).} }
\date{\today}

\begin{document}
\maketitle
\begin{abstract}
Related to a bounded-diameter bipartite analogue of the
Henderson--Ryser~\cite{HE} conjecture,
DeBiasio, Kamel, McCourt, and Sheats~\cite{DKMS} proved that the vertices of every
$2$-colored complete bipartite graph can be covered by two
monochromatic subgraphs, each of diameter at most four. We improve this
bound on the diameter to the best possible value of {\em three}.
\end{abstract}

\section{Introduction}
For $2$-colorings of a graph $G$, we use the colors red (or $1$) and blue (or $2$), whichever notation is more convenient.
We denote the red subgraph by $G_1$ and the blue subgraph by $G_2$.
For example, $N_{G_1}(v)$ (or $N_1(v)$ when $G$ is clear from context) denotes the set of red neighbors of $v$.
We write $K_{m,n}$ for the complete bipartite graph with partite sets of sizes $m$ and $n$, respectively. We write $G[A,B]$ for a bipartite graph with partite sets $A$ and $B$, and $K[A,B]$ for the complete bipartite graph with partite sets $A$ and $B$.
Given a graph $G$ and a set $X\subseteq V(G)$, we say that the induced subgraph $G[X]$ has diameter at most $k$ if the distance between any two vertices of $X$ is at most $k$ in $G[X]$.

The famous Henderson--Ryser conjecture~\cite{HE} (usually referred
to simply as Ryser's conjecture) says that, for every integer
$r\geq2$, the vertex set of every $r$-edge-colored complete graph can
be covered by $r-1$ monochromatic connected components.  The case
$r=2$ is elementary, and the cases $3\leq r\leq5$ were proved by
Tuza~\cite{TU1,TU2}.  The conjecture remains open for $r\geq6$.

The following bipartite analogue of the Henderson--Ryser conjecture was proposed by Gy\'arf\'as and
Lehel; see~\cite{GY,CFGLT}.

\begin{conjecture}\label{conjbip}
Let $r$ be an integer with $r\geq 2$.  In every $r$-edge-coloring of a complete bipartite graph $G$, the set $V(G)$ can be covered by the vertex sets of $2r-2$ monochromatic connected components.
\end{conjecture}

It is straightforward to prove Conjecture~\ref{conjbip} with $2r-1$
in place of $2r-2$.  On the other hand, the number $2r-2$ cannot be
decreased; see~\cite[Proposition 1]{CFGLT}.  Chen, Fujita, Gy\'arf\'as,
Lehel, and T\'oth proved Conjecture~\ref{conjbip} for
$2\leq r\leq5$~\cite{CFGLT}; the conjecture remains open for
$r\geq6$.

Mili\v{c}evi\v{c}~\cite{MI2} proposed a strengthening of the Henderson--Ryser conjecture in which
the monochromatic covering subgraphs are required to have diameter bounded by a function of $r$ (and DeBiasio, Kamel,
McCourt, and Sheats \cite[Conjecture~21]{DKMS} proposed a stronger version where $f(r)$ is a constant not depending on $r$).  Pokrovskiy~\cite{P} recently proved that, for each fixed $r$, the Henderson--Ryser conjecture and Mili\v{c}evi\v{c}'s conjecture are equivalent.

There has recently been significant interest in bounded-diameter
versions of the Henderson--Ryser conjecture; see
\cite{DGHS,DKMS,EMMP,GS,GSW,MI1,MI2,P}.  In particular, DeBiasio, Kamel, McCourt, and Sheats \cite[Conjecture~21]{DKMS} proposed the following bounded-diameter strengthening of Conjecture \ref{conjbip}.

\begin{conjecture}\label{conjbipdiam}
For every $r\geq 2$, there is a smallest integer $f(r)$ such that in every $r$-edge-coloring of a complete bipartite graph $G$, the set $V(G)$ can be covered by the vertex sets of $2r-2$ monochromatic connected components, each of diameter at most $f(r)$.
\end{conjecture}

The finiteness of $f(r)$ is currently known for $r=2,3$:
DeBiasio, Kamel, McCourt, and Sheats \cite[Theorems~41 and~44]{DKMS} proved
$3\leq f(2)\leq 4$ and $3\leq f(3)\leq 6$ respectively.  We
completely settle the two-color case by proving that $f(2)=3$ (for the lower bound, see Example \ref{ex:lowerbound}); a result which can be stated more explicitly as follows.

\begin{theorem}\label{thm:main}
Every $2$-colored complete bipartite graph can be covered by two monochromatic subgraphs, each of diameter at most 3.  Consequently $f(2)=3$.
\end{theorem}

English, McCourt, Mattes, and Phillips~\cite{EMMP} proved that every $2$-colored complete multipartite graph with at least three parts can be covered by two monochromatic subgraphs of diameter at most three.
Theorem~\ref{thm:main} implies their result and gives the following
general statement.

\begin{corollary}\label{multipartitecover}
Every $2$-colored complete multipartite graph with at least two parts
admits a cover by two monochromatic subgraphs, each of diameter at most
three.
\end{corollary}

For the rest of the paper, we use the following notation.  We say that $(H_1,H_2)$ is a {\em good cover} of a $2$-colored complete bipartite graph $G=K[A,B]$ if $V(H_1)\cup V(H_2)=V(G)$, and for all $i\in [2]$, $H_i$ is a monochromatic subgraph of $G$ with diameter at most 3.

In the course of proving $f(2)\leq 4$, DeBiasio, Kamel, McCourt, and Sheats \cite[Theorem~22(i)]{DKMS} proved that any 2-colored complete bipartite graph can be covered by two monochromatic {\em trees} of diameter at most 4.
A {\em double star} is obtained by joining the centers of two stars (and we consider a star to be a double star). Every tree of diameter at most 3 is a double star. One may wonder if our result can be improved by requiring that the good cover consist of double stars, but it cannot.  Random colorings provide such examples (see \cite[Example 50]{DKMS}); we also give an explicit example in Section \ref{exes}.

The following useful observation shows that general bipartite graphs of diameter at most 3 can be considerably more complicated than double stars.  In particular, every dense enough random bipartite graph will have diameter 3 with high probability.
\begin{observation}\label{equicond} A bipartite graph $G[A,B]$ with at least one edge has diameter at most $3$ if and only if the distance between any two distinct vertices in the same partite class is two.
\end{observation}
\begin{proof}
If the distance between two vertices in the same partite class is greater than $2$, then it is at least $4$, so the condition is necessary. For sufficiency, suppose that $a\in A$ and $b\in B$ are at distance at least $5$. Along a shortest path
\[
a,b_1,a_1,b_2,a_2,\dots,b,
\]
the vertices $a$ and $a_2$ have distance greater than $2$, a contradiction.
\end{proof}

We note that our proof of Theorem \ref{thm:main} actually shows that every 2-colored complete bipartite graph has a good cover with the additional property that at least one of the subgraphs in the cover has radius at most two (has a vertex with distance 2 from all other vertices). In fact, in all but Theorem \ref{type1q1} (Type 1 strong colorings), we prove that one can get a good cover with the additional property that at least one of the subgraphs in the cover is a double star (which has radius at most two and diameter at most three).  Furthermore, Example \ref{ex:ramsey} shows that it is not possible to guarantee that both subgraphs have radius at most two and diameter at most three. We leave the following question as a remaining challenge.

\begin{question}
Is it true that every $2$-colored complete bipartite graph can be covered by monochromatic subgraphs $H_1$ and $H_2$ where $H_1$ is a double star and $H_2$ has diameter at most three? 
\end{question}

\section{Three examples}\label{exes}

The following examples show that we cannot strengthen the definition of good cover in any of five different natural ways.

The first example, which appears as \cite[Example 48]{DKMS}, gives the lower bound $f(2)\ge3$.  We repeat it here for the convenience of the reader.

\begin{example}\label{ex:lowerbound}
Let $G=K[A,B]$, where
\[
A=\{a_1,a_2,a_3\}\text{ and }B=\{b_1,b_2,b_3,b_4\}.
\]
Color the edges of the path
$b_1,a_1,b_2,a_2,b_3,a_3,b_4$
red and all remaining edges blue. The blue graph is also a path $b_2,a_3,b_1,a_2,b_4,a_1,b_3.$
In any cover of $V(G)$ by two monochromatic subgraphs, one of the subgraphs contains at least four vertices and since it is a subgraph of a path, it has diameter at least 3.  Thus $f(2)\ge 3$.
\end{example}

The next example shows that we cannot require that both subgraphs in a good cover are monochromatic trees (i.e.~double stars), have radius at most two, or have different colors.

\begin{example}\label{ex:ramsey}
Consider $G=K[A,B]$ with
\[
A=\{1,2,3,4\} \text{ and } B=\{b_{ij}:1\le i<j\le4\}.
\]
For each $1\le i<j\le4$, color $(i,b_{ij})$ and $(j,b_{ij})$ red, and color every other edge blue. This self-complementary $2$-coloring of $K_{4,6}$ is familiar from bipartite Ramsey theory. It contains no monochromatic $K_{2,2}$, whereas neither $K_{3,7}$ nor $K_{5,5}$ admits a $2$-coloring with this property. It is also a strong Type I coloring in the terminology of Section~\ref{sec:typeI}, and it has the following two properties.
\begin{itemize}
\item[(i)] $G$ has no good cover by two double stars.
\item[(ii)] Every good cover of $G$ uses two subgraphs of the same color; a star and a six-cycle form such a cover.
\end{itemize}
\end{example}

\begin{proof}
For (i), a monochromatic double star has at most five vertices, and two monochromatic double stars with five vertices each must intersect. Therefore two monochromatic double stars cannot cover all ten vertices of $G$.

To prove (ii), let $(H_1,H_2)$ be a good cover of $G$ and set
\[
B_1=V(H_1)\cap B\text{ and }B_2=V(H_2)\cap B.
\]
By the definition of $H_i$, any two vertices $b,b'\in B_i$ are joined by a monochromatic path $b,a,b'$ with $a\in A$. Thus the index pairs corresponding to the vertices in $B_i$ form a pairwise intersecting family of $2$-element subsets of $\{1,2,3,4\}$; this remains true when the path is blue. Such a family is either a star, with at most three members, or a triangle. Since $B=B_1\cup B_2$, both $B_1$ and $B_2$ have exactly three elements and are disjoint. Their index pairs therefore partition the six pairs of $\{1,2,3,4\}$ into two pairwise intersecting families. The only such partition consists of a triangle and a star, which implies that the corresponding covering subgraphs have the same color (if their colors differed, then, depending on which subgraph was red, either the center of the star would be uncovered or at most two of the other three vertices would be covered). This proves (ii), and a star and a six-cycle of that color form a good cover of $G$.
\end{proof}

Finally the last example shows that we cannot require that both subgraphs in a good cover have the same color.

\begin{example}\label{ex:differentcolors}
Delete vertex $4$ from $A$ in Example~\ref{ex:ramsey} to obtain $G'$. Every good cover of $G'$ uses two subgraphs of different colors.
\end{example}

\begin{proof}
The vertices $b_{14},b_{24},b_{34}$ form a red-independent triangle, so two red subgraphs cannot cover all three of these vertices. Similarly, $b_{12},b_{13},b_{23}$ form a blue-independent triangle, so two blue subgraphs cannot cover all three of these vertices. Consequently, every good cover uses one subgraph of each color.
\end{proof}

\begin{remark}
We note that by blowing up vertices, all of the examples above can be made arbitrarily large while maintaining their essential properties.
\end{remark}

\section{Properties of a minimal counterexample}\label{sec:reductions}

Given $i\in [2]$, a pair of vertices $x,y$ in same partite class of $G$ is {\em $i$-independent} if $N_i(x)\cap N_i(y)=\emptyset$.  The pair is {\em doubly independent} if it is independent in both colors.
Similarly, $x,y$ is {\em $i$-covered} if $N_i(x)\cap N_i(y)\neq \emptyset$.  The pair is {\em doubly covered} if it is covered in both colors.  Furthermore, a vertex is {\em doubly covered} if it is incident to a doubly covered pair.  We say that $x,y$ is {\em mixed} if it is independent in one color and covered in the other.  If we additionally have a set $S$, we say that $x,y$ is {\em $i$-covered from $S$} if $N_i(x)\cap N_i(y)\cap S\neq\emptyset$.  Finally, a partite class is {\em $2$-reachable} in color $i\in [2]$ if every pair of its vertices is $i$-covered.

A vertex $v$ in a $2$-colored complete bipartite graph $G=K[A,B]$ is called a {\em special vertex} if all edges incident to $v$ have the same color.

\begin{lemma}\label{specialcover}
If a $2$-colored complete bipartite graph has a special vertex, then it has a good cover consisting of a double star and a star.
\end{lemma}

\begin{proof}
Let $v$ be a special vertex and assume without loss of generality that all edges incident to $v$ are red. Choose any vertex $w$ in the opposite partite class. The red double star with base edge $(v,w)$ covers every vertex in the partite class opposite $v$ and every red neighbor of $w$. The blue star formed by the blue edges incident to $w$ covers every remaining vertex. Together, these two subgraphs form a good cover.
\end{proof}

Distinct vertices $v,w$ of $K[A,B]$ in the same
partite class are {\em equivalent} if $N_1(v)=N_1(w)$; consequently, $N_2(v)=N_2(w)$ as well. A coloring with no equivalent vertices is called {\em reduced}.

\begin{lemma}\label{equivreduction}
If $v,w$ are distinct equivalent vertices of $G$ and $G-v$ has a good cover, then $G$ has a good cover.
\end{lemma}

\begin{proof}
Suppose that $v,w\in A$, and let $(H_1,H_2)$ be a good cover of $G-v$. Assume without loss of generality that $w\in V(H_1)$ and that $H_1$ is red. If $w$ is the only vertex of $H_1$ in $A$ and $H_1$ contains a vertex of $B$, then $H_1$ is a nontrivial star and, because $v,w$ are equivalent, adding $v$ extends it to a red complete bipartite graph. If $H_1$ consists only of $w$, choose any $b\in B$ and replace $H_1$ by the monochromatic two-edge star on $\{v,w,b\}$. Otherwise, every vertex of $V(H_1)\cap A$ is joined to $w$ by a two-edge red path. Since $v,w$ are equivalent, the midpoints of these paths are also red-adjacent to $v$, so $(H_1\cup\{v\},H_2)$ is a good cover of $G$. The other cases are symmetric.
\end{proof}

\begin{lemma}\label{nested}
If there exist distinct vertices $x,y$ in the same partite set and $i\in [2]$ such that
\[
N_i(y)\subsetneq N_i(x),
\]
then $G$ has a good cover consisting of double stars of opposite colors.
\end{lemma}

\begin{proof}
Without loss of generality, $x,y\in A$. Choose $b\in N_i(x)\setminus N_i(y)$. The double star in color $i$ with base edge $(x,b)$ covers
\[
N_i(b)\cup N_i(x),
\]
while the double star in color $3-i$ with base edge $(y,b)$ covers
\[
N_{3-i}(b)\cup N_{3-i}(y).
\]
The two sets cover $A$ because $N_i(b)\cup N_{3-i}(b)=A$, and they cover $B$ because $N_i(y)\subseteq N_i(x)$ implies $N_i(x)\cup N_{3-i}(y)=B$. The case $x,y\in B$ is symmetric.
\end{proof}

\begin{proposition}\label{reachabilitycover}
For all $i\in [2]$, if neither partite class is $2$-reachable in color $i$, then $G$ has a good cover by two double stars of color $3-i$.
\end{proposition}

\begin{proof}
Without loss of generality, suppose neither partite class is $2$-reachable in blue.
Since $A$ is not $2$-reachable in blue, there are vertices $a_1,a_2\in A$ with no blue path of length two between them. Hence
\[
N_2(a_1)\cap N_2(a_2)=\emptyset,
\]
and therefore
\begin{equation}\label{fed}
N_1(a_1)\cup N_1(a_2)=B.
\end{equation}
Similarly, since $B$ is not $2$-reachable in blue, there are vertices $b_1,b_2\in B$ with no blue path of length two between them. The vertices $b_1,b_2$ cannot both lie in $N_2(a_1)$ or both lie in $N_2(a_2)$, since they would then be joined by a blue path of length two through $a_1$ or $a_2$. Therefore, after relabeling if necessary, $b_1\in N_1(a_1)$ and $b_2\in N_1(a_2)$. As in \eqref{fed},
\[
N_1(b_1)\cup N_1(b_2)=A.
\]
The two red double stars with base edges $(a_1,b_1)$ and $(a_2,b_2)$ provide the desired cover.
\end{proof}

We prove Theorem \ref{thm:main} by classifying and eventually excluding potential counterexamples.  A counterexample $G=K[A,B]$ to the statement $f(2)\leq 3$ is called {\em minimal} if $|A|+|B|$ is as small as possible.

\begin{theorem}\label{mincount}
Every minimal counterexample $G=K[A,B]$ to the statement $f(2)\leq 3$ has the following properties.
\begin{enumerate}
\item[(i)] $G$ has no special vertex.
\item[(ii)] $G$ has no equivalent vertices.
\item[(iii)] For all distinct $x,y\in V(G)$ and all $i\in[2]$, $N_i(x)\not\subseteq N_i(y)$.
\item[(iv)] Either every pair of vertices in one partite class is doubly covered (a Type I coloring), or every pair of vertices in one partite class is red-covered and every pair of vertices in the other is blue-covered (a Type II coloring).
\end{enumerate}
\end{theorem}

\begin{proof}
(i) This follows immediately from Lemma \ref{specialcover}.

(ii) If $G$ had equivalent vertices $v,w$, then the minimality of $G$ would give a good cover of $G-v$, and Lemma \ref{equivreduction} would extend it to $G$, a contradiction.

(iii) Proper inclusion is ruled out by Lemma \ref{nested}, while equality is ruled out by (ii).

(iv) Proposition \ref{reachabilitycover}, applied in each color, shows that at least one partite class is $2$-reachable in red and at least one partite class is $2$-reachable in blue. If the same partite class can be chosen for both colors, the coloring is Type I. Otherwise the two choices lie in opposite partite classes, and the coloring is Type II.
\end{proof}

\section{Proof of Theorem \ref{thm:main}}\label{mainproof}

We now study Type I and Type II colorings separately and find a good cover in each case, thereby completing the proof of Theorem \ref{thm:main}. For Type II colorings, we prove the stronger conclusion that one of the covering subgraphs can be chosen to be a double star.  We begin with a lemma which will be used in both cases.

\begin{lemma}[Base edge covering]\label{rootedcover}
Let $(a,b)$ be a red edge of $G=K[A,B]$, where $a\in A$ and $b\in B$, and set
\[
A_1=N_1(b),~A_2=N_2(b),~B_1=N_1(a),~\text{and}~B_2=N_2(a).
\]
Assume that $A_2,B_2\ne\emptyset$ and that
\[
N_2(x)\cap B_2\ne\emptyset\text{ for every }x\in A_2,
\]
and
\[
N_2(y)\cap A_2\ne\emptyset\text{ for every }y\in B_2.
\]
Under these assumptions, $G$ has a good cover consisting of the red double star with
base edge $(a,b)$ and the blue subgraph
\(
G_2[A_2\cup\{a\},B_2\cup\{b\}].
\)
\end{lemma}

\begin{proof}
The red double star covers $\{a,b\}\cup A_1\cup B_1$. In the blue subgraph, any
two vertices of $A_2$ have a common neighbor $b$, while any two
vertices of $B_2$ have a common neighbor $a$. The two assumptions
ensure that $a$ is blue-covered with every vertex of $A_2$ and that
$b$ is blue-covered with every vertex of $B_2$. Hence every two
vertices in the same partite class have distance two, so the blue
subgraph has diameter at most three.
\end{proof}

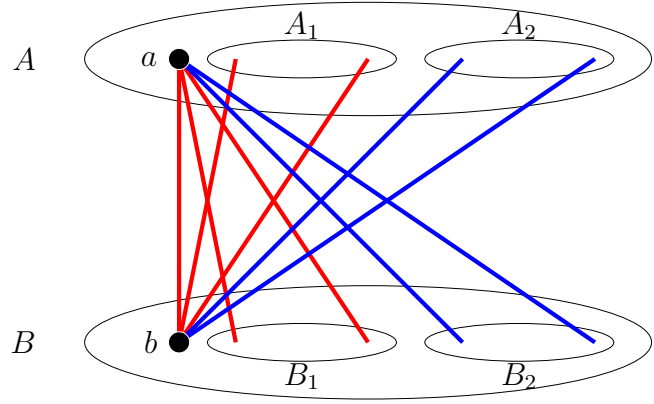
\begin{figure}[ht]
	\centering
        \begin{tikzpicture}[scale=2.5]
				\coordinate (a) at (-0.5,0);
				\coordinate (b) at (-0.5,1.5);		
				\coordinate (c) at (-0.2,0);
				\coordinate (d) at (0.5,0);
                \coordinate (c2) at (-0.2,1.5);
				\coordinate (d2) at (0.5,1.5);
               	\coordinate (c1) at (1.0,0);
				\coordinate (d1) at (1.7,0);
    	        \coordinate (c3) at (1.0,1.5);
				\coordinate (d3) at (1.7,1.5);
				\coordinate (e) at (0.5,1.5);
			    \coordinate (f) at (0.5,0);
         		\coordinate (g) at (0.15,0);
			    \coordinate (h) at (1.3,0);
	            \coordinate (i) at (0.15,1.5);
			    \coordinate (j) at (1.3,1.5);

    \def\bigrad{1.5}
    \def\smrad{0.3}
           \draw[ellip] (e) ellipse ({\bigrad} and {\smrad});
         \draw[ellip] (f) ellipse ({\bigrad} and {\smrad});
        \draw[ellip] (g) ellipse ({0.5} and {0.1});
         \draw[ellip] (h) ellipse ({0.5} and {0.1});
         \draw[ellip] (i) ellipse ({0.5} and {0.1});
         \draw[ellip] (j) ellipse ({0.5} and {0.1});
						\begin{pgfonlayer}{forefore}
					\draw (a) node[vtx,label={left:$b$}] {};
					\draw (b) node[vtx,label={left:$a$}] {};
                    \node [left] at (-1.2,1.5) {$A$};
                    \node [left] at (-1.2,0) {$B$};
                    \node [above] at (0.15,1.55) {$A_1$};
                    \node [above] at (1.3,1.55) {$A_2$};
                   \node [below] at (0.15,-0.05) {$B_1$};
                    \node [below] at (1.3,-0.05) {$B_2$};

  					\end{pgfonlayer}
				\begin{pgfonlayer}{fore}
					\draw[edg,red] (b) -- (c);
					\draw[edg,red] (b) -- (d);
				      \draw[edg,red] (a) -- (c2);
					\draw[edg,red] (a) -- (d2);
	                \draw[edg,blue] (b) -- (c1);
					\draw[edg,blue] (b) -- (d1);
                    \draw[edg,blue] (a) -- (c3);
					\draw[edg,blue] (a) -- (d3);
					\draw[edg,red] (a) -- (b);
							\end{pgfonlayer}
			\end{tikzpicture}
\caption{The structure of the base edge covering.}
            \label{cover1}
	\end{figure}

\subsection{Type I colorings}\label{sec:typeI}

Throughout this subsection, whenever $G=K[A,B]$ has a Type I coloring, we assume without loss of generality that every pair in $A$ is doubly covered. For such a coloring, let $G_B$ be the graph with vertex set $B$ in which two vertices are adjacent exactly when they form a doubly independent pair.

\begin{lemma}\label{matching}
If $G=K[A,B]$ is a reduced Type I coloring, then $\Delta(G_B)\le1$; equivalently, $E(G_B)$ is a matching.
\end{lemma}

\begin{proof}
It suffices to show that $G_B$ contains no path on three vertices $b_1,b_2,b_3$. Since $(b_1,b_2)$ is doubly independent, $A$ is partitioned into sets $A_1,A_2$ such that $[b_1,A_1]$ and $[b_2,A_2]$ are blue stars, while $[b_1,A_2]$ and $[b_2,A_1]$ are red stars. Applying the same argument to the doubly independent pair $(b_2,b_3)$ shows that $b_1$ and $b_3$ are equivalent, contradicting the assumption that $G$ is reduced.
\end{proof}

We call a reduced Type I coloring $G=K[A,B]$ {\em strong} if $E(G_B)$ is a perfect matching of $B$, with edges
\begin{equation}\label{eq1}
e_1=(b_1,b_2),\dots,e_m=(b_{2m-1},b_{2m})
\end{equation}
and, for every $1\le i<j\le m$, all four pairs with one endpoint in $e_i$ and the other in $e_j$ are doubly covered from $A$. Let $G(m)$ be the family of strong colorings with $|B|=2m$.
Example \ref{ex:ramsey} is an interesting member of $G(3)$. We next show that the proof for Type I colorings can be reduced to the strong case.

\begin{lemma}\label{strong}
Let $G=K[A,B]$ be a reduced Type I coloring. If $G$ is not strong, then there is a good cover consisting of a double star in one color and a subgraph having diameter at most 3 in the other.
\end{lemma}

\begin{proof}
By Lemma \ref{specialcover}, we may assume that $G$ has no special vertex.

By Lemma \ref{matching}, $E(G_B)$ is a matching. Assume first that it is not a perfect matching of $B$, and choose $b\in B$ not incident to any doubly independent pair.
If there are vertices $b',b^{''}\in B$ such that $(b,b')$ is blue-independent and $(b,b^{''})$ is red-independent, then since $b$ is not incident to a doubly independent pair we have $b'\ne b^{''}$ and
\[
N_1(b^{''})\subseteq N_2(b)\subseteq N_1(b').
\]
Since $G$ is reduced, $N_1(b^{''})\subsetneq N_1(b')$ and thus Lemma \ref{nested} gives the desired cover.  We may therefore assume that no such vertices $b',b^{''}$ exist. Without loss of generality, suppose $b$ is not incident to any blue-independent pair. Since $b$ is not special, it has a red neighbor $a$. Define
\[
A_1=N_1(b),~A_2=N_2(b),~B_1=N_1(a),~\text{and}~B_2=N_2(a).
\]
Since neither $a$ nor $b$ is special, the sets $A_2$ and $B_2$ are nonempty. For every $a'\in A_2$, the pair $(a,a')$ is blue-covered because all pairs in $A$ are doubly covered. For every $b'\in B_2$, the pair $(b,b')$ is blue-covered by the choice of $b$. The middle vertices of these paths necessarily lie in $B_2$ and $A_2$, respectively. Hence the hypotheses of Lemma \ref{rootedcover} are satisfied and thus we obtain the required cover (see Figure \ref{cover1}).

It remains to consider the case in which $E(G_B)$ is a perfect matching of $B$. Since the coloring is not strong, at least one of the four pairs between some $e_i$ and $e_j$, say $(b_{2i},b_{2j})$, is not doubly covered from $A$; say it is red-independent.  Since $(b_{2i-1},b_{2i})$ is blue-independent, we have
\[
N_1(b_{2j})\subseteq N_2(b_{2i})\subseteq N_1(b_{2i-1}).
\]
Since $G$ is reduced, $N_1(b_{2j})\subsetneq N_1(b_{2i-1})$ and thus Lemma \ref{nested} gives the desired cover.
\end{proof}

Finally we handle strong colorings.

\begin{theorem}\label{type1q1}
If $G=K[A,B]\in G(m)$, then $G$ has a good cover.
\end{theorem}
\begin{proof}
Let $a\in A$ be arbitrary. Label the endpoints of the matching edges so that
\[
B_1=N_1(a)=\{b_1,b_3,\dots,b_{2m-1}\}
\quad\text{and}\quad
B_2=N_2(a)=\{b_2,b_4,\dots,b_{2m}\}.
\]
Since each pair $(b_{2i-1},b_{2i})$ is doubly independent, every vertex of $A$ sends edges of opposite colors to its two endpoints. Define a graph $H^a$ on $A$ by joining $u,v\in A$ when they have no common red neighbor in $B_1$. The vertex $a$ is isolated in $H^a$ since every pair containing $a$ is red-covered, and every red neighbor of $a$ lies in $B_1$.

Suppose first that $H^a$ is bipartite, with bipartition $X_1,X_2$. Since $a$ is isolated, we may assume that $a\in X_1$. Any two vertices of $X_1$ are red-covered from $B_1$, and any two vertices of $B_1$ are red-covered through $a$. Hence $G_1[X_1,B_1]$ has diameter at most three. Any two vertices of $X_2$ have a common red neighbor $b_{2i-1}\in B_1$ and therefore have the common blue neighbor $b_{2i}\in B_2$. Also, each $x\in X_2$ is blue-covered with $a$, necessarily from $B_2$, and any two vertices of $B_2$ are blue-covered through $a$. Thus $G_2[X_2\cup\{a\},B_2]$ has diameter at most three, giving a good cover.

Suppose now that $H^a$ is not bipartite. Let
$C=u_1u_2\dots u_{2k+1}u_1$
be a shortest odd cycle in $H^a$, with indices interpreted modulo $2k+1$. In particular, $C$ is induced.

\begin{claim}
The graph $G_1[V(C),B_2]$ has diameter at most three.
\end{claim}

\begin{claimproof}
We first show that every two distinct vertices of $C$ are red-covered from $B_2$. If $u_i,u_j$ are adjacent on $C$, then they have no common red neighbor in $B_1$. Since every pair in $A$ is red-covered, they must be red-covered from $B_2$. If $u_i,u_j$ are not adjacent, then their successors $u_{i+1},u_{j+1}$ are also nonadjacent in $H^a$, so they have a common red neighbor $b_{2\ell-1}\in B_1$. Hence both are blue-adjacent to its matching partner $b_{2\ell}\in B_2$. Since $u_iu_{i+1}\in E(H^a)$, the vertex $b_{2\ell-1}$ cannot also be red-adjacent to $u_i$. Thus $u_i$ is blue-adjacent to $b_{2\ell-1}$ and red-adjacent to $b_{2\ell}$. Similarly, $u_j$ is red-adjacent to $b_{2\ell}$.

Next, every vertex $b_{2i}\in B_2$ has at least $k+1$ red neighbors on $C$. Otherwise it has at least $k+1$ blue neighbors on the cycle, two of which are consecutive. Their edges to $b_{2i-1}$ are then both red, contradicting the fact that consecutive vertices of $C$ are adjacent in $H^a$. Since two subsets of a $(2k+1)$-element set, each of size at least $k+1$, must intersect, every two vertices of $B_2$ have a common red neighbor on $C$. The result now follows from Observation \ref{equicond}.
\end{claimproof}

Choose a maximal set $A_2\supseteq V(C)$ subject to $G_1[A_2,B_2]$ having diameter at most three, and set $A_1=A\setminus A_2$. Since $a$ has no red neighbor in $B_2$, we have $a\in A_1$. Therefore every two vertices of $B_1$ are red-covered through $a$.

It remains to show that every two vertices of $A_1$ have a common red neighbor in $B_1$. Let $u,v\in A_1$. By the maximality of $A_2$, there exist $u',v'\in A_2$ such that $(u,u')$ and $(v,v')$ are not red-covered from $B_2$. Since $G_1[A_2,B_2]$ has diameter at most three, the vertices $u',v'$ have a common red neighbor $b_{2i}\in B_2$. It follows that $u$ and $v$ are both blue-adjacent to $b_{2i}$ and hence both red-adjacent to $b_{2i-1}\in B_1$. Thus $G_1[A_1,B_1]$ has diameter at most three, and the two red subgraphs $G_1[A_1,B_1]$ and $G_1[A_2,B_2]$ form a good cover.
\end{proof}

\subsection{Type II colorings}\label{sec:typeII}
Throughout this subsection, whenever $G=K[A,B]$ has a Type II coloring, we assume without loss of generality that every pair in $A$ is red-covered and every pair in $B$ is blue-covered.

\begin{lemma}\label{type2}
Let $G=K[A,B]$ have a Type II coloring.  If every vertex in one of the partite classes belongs to a mixed pair, then either $G_1$ or $G_2$ has diameter at most three.
\end{lemma}

\begin{proof}
Without loss of generality suppose that every vertex in $A$ belongs to a mixed pair.  We claim that $B$ has no mixed pairs. Suppose, to the contrary, that $(b_1,b_2)$ is a mixed pair; thus it is blue-covered and red-independent.

Let $a\in A$ be a common blue neighbor of $b_1,b_2$. Since every vertex in $A$ belongs to a mixed pair, there exists $a'\in A$ such that $(a,a')$ is mixed. Every pair in $A$ is red-covered, so $(a,a')$ is blue-independent. Consequently, $a'$ is red-adjacent to both $b_1$ and $b_2$, contradicting the red-independence of $(b_1,b_2)$.

Therefore $B$ has no mixed pairs. Since every pair in $B$ is blue-covered, every pair in $B$ is also red-covered. Every pair in $A$ is red-covered by the Type II assumption, so Observation \ref{equicond} shows that $G_1$ has diameter at most three.
\end{proof}

We now use this lemma to obtain a good cover for every Type II coloring.

\begin{theorem}\label{type2cover}
If $G=K[A,B]$ has a Type II coloring, then $G$ has a good cover by a double star and a subgraph of diameter at most three.
\end{theorem}

\begin{proof}
By Lemma \ref{specialcover}, we may assume that $G$ has no special vertex. By Lemma \ref{type2}, we may also assume that some vertex $a\in A$ is not incident to any mixed pair in $A$ (note that if $G_i$ has diameter at most 3, then $G_i$ together with any edge gives the desired conclusion). Choose $b\in N_1(a)$ and define
\[
A_1=N_1(b),~A_2=N_2(b),~B_1=N_1(a),~\text{and}~B_2=N_2(a).
\]
Since $G$ has no special vertex, $A_2\neq \emptyset$ and $B_2\neq \emptyset$. For every $a'\in A_2$, the pair $(a,a')$ is not mixed and hence is blue-covered; every common blue neighbor of this pair lies in $B_2$. Similarly, for every $b'\in B_2$, the pair $(b,b')$ is blue-covered, and every common blue neighbor lies in $A_2$. Therefore the hypotheses of Lemma \ref{rootedcover} hold, and that lemma gives the desired cover (see Figure \ref{cover1}).
\end{proof}

\begin{proof}[Proof of Theorem~\ref{thm:main}]
Suppose that a minimal counterexample $G$ exists. By Theorem \ref{mincount}~(ii) and (iv), $G$ is reduced and has either a Type I or a Type II coloring. If $G$ has a Type I coloring, then Lemma \ref{strong} applies when $G$ is not strong, while Theorem \ref{type1q1} applies when $G$ is strong. The Type II case is ruled out by Theorem \ref{type2cover}. Hence no minimal counterexample exists, so every $2$-colored complete bipartite graph has a good cover and $f(2)\le3$. Example \ref{ex:lowerbound} gives $f(2)\ge3$ and thus $f(2)=3$.
\end{proof}

\end{document}